\documentclass{amsart}
\usepackage{graphicx}
\usepackage{amssymb,amscd,amsthm,amsxtra}
\usepackage{latexsym}
\usepackage{epsfig,esint}

\newtheorem{thm}{Theorem}[section]

\newtheorem{lem}[thm]{Lemma}
\newtheorem{prop}[thm]{Proposition}
\theoremstyle{definition}
\newtheorem{defn}[thm]{Definition}
\theoremstyle{remark}
\newtheorem{rem}[thm]{Remark}
\numberwithin{equation}{section}

\newcommand{\R}{\mathbb R}

\newcommand{\be}{\begin{equation}}
\newcommand{\ee}{\end{equation}}

\newcommand{\eps}{\epsilon}

\newcommand{\p}{\partial}

\newcommand{\comment}[1]{}

\begin{document}

\title[Two-phase problems and the Bellman equation in 2D]{Two-phase anisotropic free boundary problems and applications to the Bellman equation in 2D}

\author{L. Caffarelli}
\address{Department of Mathematics, University of Texas at Austin, Austin, TX 78712, USA}\email{\tt caffarel@math.utexas.edu}
\author{D. De Silva}
\address{Department of Mathematics, Barnard College, Columbia University, New York, NY 10027, USA}
\email{\tt  desilva@math.columbia.edu}
\author{O. Savin}
\address{Department of Mathematics, Columbia University, New York, NY 10027, USA}\email{\tt  savin@math.columbia.edu}
\thanks{L.~C. is supported by NSF grant DMS-1500871. D.~D. is supported by NSF grant DMS-1301535.  O.~S.~is supported by  NSF grant DMS-1200701.}

\maketitle

\begin{abstract} We prove Lipschitz continuity of solutions to a class of rather general two-phase anisotropic free boundary problems in 2D and we classify global solutions. As a consequence, we obtain $C^{2,1}$ regularity of solutions to the Bellman equation in 2D.
\end{abstract}
\section{Introduction}

One of the basic fully nonlinear 2nd order PDE is the Bellman equation, which is the equation of dynamic programming for certain optimally controlled stochastic systems. In the case of two operators the equation reads:
\begin{equation}\label{BE}F(D^2 v) := Min\{L_1 v, L_2 v\}=0,\end{equation}
where $L_i$ are constant coefficient elliptic operators
$$L_i v : = tr(A_i D^2v), \quad i=1,2 \quad \lambda I \leq A_i \leq \Lambda I.$$

The operator $F$ is concave and solutions to \eqref{BE} satisfy the $C^{2,\alpha}$ interior estimates of Evans and Krylov (\cite{E,K}). 
In fact, the question of interior regularity for problem \eqref{BE} in the case of  two operators was first settled by Brezis and Evans in \cite{BE}.

The motivation for this note was to investigate further regularity of $v$ and qualitative properties of  the free boundary 
$$\Gamma:=\{ x \in B_1 | \quad L_1v(x)=L_2v(x)\},$$
which is a closed set in $B_1$ due to the H\"older continuity of $D^2v$.

Our main results hold in 2D and they read as follows.

\begin{thm}\label{T0} Let $v$ satisfy \eqref{BE} in $B_1 \subset \R^2.$
Assume $0 \in \Gamma$. Then one of the following alternatives holds.

1) $\Gamma$ is a smooth curve in a neighborhood of $0$ and in this case $v$ is $C^{2,1}$ near 0 and has an expansion
$$v(x)=Q + \gamma \left((x \cdot \nu)^+ \right)^3 + O(|x|^{3+\alpha}) ,$$
with $Q$ a third order polynomial and $\nu$ the normal to $\Gamma$ at $0$.

2) $v$ is pointwise $C^{3,\alpha}$ at $0$ i.e. has an expansion
    $$v=P + O(|x|^{3+\alpha}), $$
    with $P$ a polynomial of degree 2 (all third derivatives vanish at 0).
\end{thm}

The coefficients of $Q$, $P$ (and $\gamma$) and the constant in $O(|x|^{3+\alpha})$ are bounded by $C \|v\|_{L^\infty(B_1)}$ with $C$ a constant depending only on $\lambda, \Lambda$.
As a consequence we obtain the optimal regularity of solutions to Bellman equation in 2D.

\begin{thm}
Assume $v$ satisfies \eqref{BE} and $|v| \le 1$ in $B_1 \subset \R^2$. Then
$$\|v\|_{C^{2,1}(B_{1/2})} \le C (\lambda,\Lambda).$$
\end{thm}

Our strategy consists in viewing this problem as a so-called two phase free boundary problem. We obtain regularity results for such type of problems that translate in our main Theorems above. These results are in fact the core of our paper and they  are interesting in their own. 

Precisely, 
assume that $u: B_1 \to \R$, $B_1 \subset \R^2$ satisfies
\begin{equation}\label{TP}
\left \{
\begin{array}{lr}
 L_1 u=0 \quad \quad \mbox{in} \quad B_1^+(u):=\{u>0\}, \\
 L_2 u=0 \quad \quad \mbox{in} \quad B_1^-(u):=\{ u \le 0\}^\circ, \\
 u_\nu^+=G(u_\nu^-,\nu) \quad \quad \mbox{on} \quad F(u):=\p B_1^+(u) \cap B_1,
\end{array}
\right.
 \end{equation}
where the free boundary condition is understood in the viscosity sense (see Section 2 for the precise definition) and $u_\nu^+$ and $u_\nu^-$ denote the normal derivatives of $u^+$ and $u^-$ in the inward
direction to $B_1^+(u)$ and $B_1^-(u)$ respectively.
Here $$L_i u(x)= tr(A_i(x)D^2 u), \quad 0<\lambda I \leq A_i (x)\leq \Lambda I \quad \text{in $B_1$}$$ with $A_i$ H\"older continuous.  The function $G: \R^+ \times S^1 \to \R^+$ is continuous and it satisfies the usual ellipticity assumption
$$\mbox{ $G(b,\nu)$ is strictly increasing in $b$} $$ 
and $G(b,\nu) \to \infty$ as $b \to \infty$ uniformly in $\nu$, i.e
\begin{equation}\label{omega}G(b,\nu) \ge \omega(b), \quad \quad \mbox{with} \quad \lim_{b \to \infty} \omega(b)=\infty.\ee

Our first main result gives the Lipschitz continuity of $u.$

\begin{thm}\label{TLip}
Let $u$ be as above. Then
$$\|\nabla u\|_{L^\infty(B_{1/2})} \le C\left(\lambda,\Lambda,\omega, \|u\|_{L^\infty(B_1)}  \right).$$
\end{thm}

The estimate in Theorem \ref{TLip} does not depend on the H\"older norm of the $A_i$ and in fact our proof carries through even when the $A_i$ are merely measurable (see Remark \ref{measurable}).

In some special cases, Theorem \ref{TLip} is well known in any dimension. In the case when $L_1=L_2=\Delta$ and for a rather general class of $G$'s,  Lipschitz continuity of a solution follows from the celebrated monotonicity formula of Alt-Caffarelli-Friedman \cite{ACF} (see also \cite{CJK, MP}). In a recent paper \cite{DK}, the authors prove Lipschitz continuity of the solution of a two-phase free boundary problem governed by the $p$-Laplacian for a special class of isotropic $G$'s. The case of two different operators is however more delicate and no other results on the Lipschitz regularity of solutions are available in the literature. Some partial regularity results of the free boundary are proved in \cite{AM, F}.

We remark that Theorem \ref{TLip} cannot hold in this generality in dimension $n \ge 3$. Indeed, say for $n=3$, it is not difficult to construct two homogeneous functions of degree less than one, that solve two different uniformly elliptic equations in complementary domains in $\R^3$ and satisfy the free boundary condition for a specific $G$.

After obtaining Theorem \ref{TLip}, we can also characterize the blow-up limits at free boundary points. Assume that $0 \in F(u)$ and let $u^*$ be a blow-up limit along a subsequence $r_k \to 0$, i.e.
$$u^*(x)=\lim_{k \to \infty} \frac{u(r_kx)}{r_k}.$$
Due to Theorem \ref{TLip} such blow-up limit functions exist and they are Lipschitz. Our second main result characterizes such blow-up limits.

\begin{thm}\label {T2}
Assume $u^*$ is as above.

Then either $u^*$ is a two plane-solution 
\begin{equation}\label{2p}
u^*=a (x\cdot \nu)^+ - b (x \cdot \nu)^- \quad \quad \mbox{with} \quad a, b>0, 
\end{equation}
or 
\begin{equation}\label{1p}
(u^*)^- \equiv 0,
\end{equation} 
which means that $u^*$ solves the one-phase problem for $L_1$.
\end{thm}

As remarked above, Theorem \ref{T2} cannot hold in this generality in dimension $n \ge 3$.

If a blow-up limit $u^*$ of $u$ is a two-plane solution then we prove smoothness of the 
free boundary of $u$ in a neighborhood of $0$, provided that $G$ is more regular. This can be achieved 
in any dimension by perturbation techniques by first analyzing the regularity of a transmission-type 
problem across $\{x_n=0\}$ for two different linear operators (see for example \cite{DFS1, DFS2, AM}).

Assume that $G$ is smooth and homogenous of degree one in $b$. Then the alternative \eqref{1p} above gives $u^* \equiv 0$. By compactness we obtain the following result.

\begin{thm}\label{T3}
Assume that $G(b,\nu)$ is homogenous of degree one in $b$, and $G(1,\nu)$ is $C^{1}$ in $\nu$. Assume that $$0 \in F(u), \quad \quad |u| \le 1 \ \text{in}  \ B_1.$$
Then there exist $a$, $b \ge 0$ and a direction $\nu$ such that 
$$\left|u(x)-\left[a(x \cdot \nu)^+-b(x \cdot \nu)^- \right] \right| \le C_0 |x|^{1+\alpha}, \quad \quad \quad \mbox{with $a=G(b,\nu)$}.$$
The constants $\alpha$, $C_0$ are universal, i.e. they depend only on $\lambda$, $\Lambda,$ the H\"older norm of the $A_i$ and $G$.

If $a \ne 0$ then $F(u)$ is $C^{1,\alpha}$ in a neighborhood of $0$ and if $a=0$ then $u$ is pointwise $C^{1,\alpha}$ at the origin. 
\end{thm} 

The strategy to prove the main theorems, Theorem \ref{TLip} and Theorem \ref{T2}, relies on two-dimensional topological arguments which involve intersecting the graph of the solution with a family of planes. These ideas go back to the work of Bernestein \cite{B} and Hopf \cite{H} and have been used more recently in \cite{DS, S}.

The paper is organized as follows. In Section 2 we prove Theorem \ref{TLip} and Theorem \ref{T2}. We describe our compactness method and prove Theorem \ref{T3} in Section 3. Section 4 is devoted to Bellman's equation and the proof of Theorem \ref{T0}. 
\section{Two-phase free boundary problems}

In this Section we provide the proofs of our main results, Theorem \ref{TLip} and Theorem \ref{T2}. First we introduce some standard definitions and prove some preliminary lemmas.

As in the Introduction (see \eqref{TP}), let $u: B_1 \subset \R^2 \to \R$ satisfy
\begin{equation}\label{FBC}
\left \{
\begin{array}{lr}
 L_1 u=0 \quad \quad \mbox{in} \quad B_1^+(u), \\
 L_2 u=0 \quad \quad \mbox{in} \quad B_1^-(u), \\
 u_\nu^+=G(u_\nu^-,\nu) \quad \quad \mbox{on} \quad F(u) \quad \mbox{when $u_\nu^->0$.}
\end{array}
\right.
 \end{equation}
We point out that we require the free boundary condition to hold only when $u_\nu^- \ne 0$. This is a weaker definition than it usually appears in the literature and it is understood in the following viscosity sense.

\begin{defn} We say that  $u$ satisfies the free boundary condition 
$$ u_\nu^+=G(u_\nu^-,\nu)$$
at a point $y_0 \in F(u)$ if for any unit vector $\nu$, there exists no function $\psi \in C^2 $ defined in a neighborhood of $y_0$ with $\psi(y_0)=0$, $\nabla \psi(y_0)=\nu$ such that either of the following holds:

1) $a \psi^+ - b \psi^- \le u$ with  $a>0$, $b > 0$  and $a> G(b, \nu)$ (i.e. $u$ is a supersolution);

2) $a \psi^+ - b \psi^- \ge u$ with $a>0$, $b > 0$ and $a < G(b, \nu)$ (i.e. $u$ is a subsolution).
\end{defn}

We only use comparison functions which cross the $0$ level set transversally and therefore have a nontrivial negative part. For this reason the free boundary condition is preserved when taking uniform limits. It is straightforward to check that a uniform limit of solutions of \eqref{FBC} satisfies \eqref{FBC} as well.

\begin{defn}A {\it two-plane solution} $p$ to \eqref{FBC} is given by
$$p(x)=p_{x_0,\nu,a,b}(x):=a((x-x_0)\cdot \nu)^+- b((x-x_0)\cdot \nu)^-,$$
for some $x_0 \in \R^2$, $\nu \in S^1$ and with $$a=G(b, \nu), \quad \quad \mbox{and} \quad a>0, b>0.$$\end{defn}

Given a function $u$ that satisfies \eqref{FBC} and a point $y$ away from the free boundary of $u$, we often consider the two-plane solution $p$ which is tangent to the graph of $u$ at $(y,u(y))$, i.e. such that $u(y)=p(y)$ and $\nabla u(y)=\nabla p(y)$. Notice that this two-plane solution might not always exist. It is well defined unless either $\nabla u(y)=0$ or if $u(y)>0$ and 
$$|\nabla u(y)| \le G (0,\nu)  \quad \quad \mbox{with} \quad 
\nu=\frac{\nabla u(y)}{|\nabla u(y)|}.$$

From the viscosity definition above and the Hopf lemma we see that a two-plane solution $p$ cannot touch $u$ by above (or below) on the free boundary $\{p=0\}$ unless $u$ and $p$ coincide. As a consequence we obtain the following maximum principle.
  
\begin{lem}[Maximum principle] \label{l1}
Let $\Omega$ be a bounded domain and $p$ a two-plane solution. If $u\le p$ (or $u \ge p$) on $\p \Omega$ then $u\le p$ (respectively $u \ge p$) in $\Omega$. 
\end{lem}

\begin{proof}

We compare $u$ with the continuous family of two-plane solutions $t \mapsto p(x+t\nu)$ which is strictly increasing in $t$. These solutions converge to $\pm \infty$ as $t \to \pm \infty$, and in view of the discussion above the maximum principle applies.
\end{proof}

We assume for simplicity that $A_i(x)$ are H\"older continuous and therefore $u$ is $C ^2$ in $\{u>0\} \cup \{ u<0\}$. If $u$ is not linear in a neighborhood of a point $z$, then $z$ belongs to the closure of the set $\{D^2u \ne 0  \}$. In this nondegeneracy set the following topological lemma holds in 2D.

\begin{lem}[Connected components]\label{l2}
Assume $u(x_0) \ne 0$ and $D^2u(x_0) \ne 0$. Let $p$ be a two-plane solution such that at $x_0$ we have $u=p$ and $\nabla u= \nabla p$. Then $\{u>p\}$ (resp. $\{u<p\}$) has two distinct connected components starting at $x_0$ which exit $B_1$. 
\end{lem}

\begin{proof}

The existence of the two components locally near $x_0$ is clear, since $u$ is a solution to an elliptic equation and the eigenvalues have opposite sign. These components cannot be compactly included in $B_1$ by the maximum principle.

Moreover, if the two components reconnect  further out inside $B_1$ then, since we are in $\R^2$, they enclose a domain compactly included in $B_1$. This domain contains a component of $\{u<p\}$ and we reach a contradiction again. 
\end{proof}

\begin{rem}\label{measurable}
We remark that in the case when the $A_i$ are merely measurable, Lemma \ref{l2} still holds in a dense subset of a neighborhood  where the function is not linear (see Lemma 1 in \cite{S}.)
\end{rem}
We are now ready to prove Theorem \ref{TLip}.

\begin{proof}[\it Proof of Theorem $\ref{TLip}$]

Assume that 
$$ \mbox{$|u| \le 1$ in $B_1$ and $|\nabla u(x_0)| \gg 1$,}$$ 
for some $x_0 \notin \{u=0\}$ near $0$. We will reach a contradiction.

If $u$ is linear in a neighborhood around $x_0$ then, by unique continuation, $u$ coincides with this linear function $\ell$ in either the set $\{\ell <0\}$ or $\{\ell >0\}$, and since $|\nabla \ell| \gg 1$ we contradict that $|u| \le 1$ in $B_1$. Thus $u$ is not linear near $x_0$, and without loss of generality we may assume that $D^2u(x_0) \ne 0$.
 
 Let $p$ be the two plane solution such that at $x_0$ we have $u=p$, $\nabla u=\nabla p$, and say for simplicity that
$$p=a x_2^+ -b x_2^-, \quad a=G(b,e_2),$$ 
with $a,b \gg 1,$ which follows from our contradiction assumption and the properties of $G$. Then, in view of Lemma \ref{l2}, $\{ u< p\}$ has two distinct components starting at $x_0$. 

Since $p$ grows fast in the $x_2$ direction and $|u| \le 1$ it follows that one component $\mathcal U$ of $\{ u<p \}$ that starts at $x_0$ is included in the thin strip
$$\mathcal U \subset \left \{-\frac 1b < x_2 < \frac 1a \right \} .$$

Let $P$ be a non self-intersecting polygonal line included in $\mathcal U$ that starts near $x_0$ and exits $B_1$ say on the right side of the strip above. Let $R$ be the rectangle
$$R:=(\frac 14,\frac 34) \times (-\frac 14, \frac 14) ,$$
and $\bar P \subset P \cap \bar R$ a part of the polygonal line $P$ which connects the two lateral sides of $R$. Then $\bar P$ splits $R$ into two components and denote by $\mathcal V$ the component on the top (see Figure \ref{fig1}).

\begin{figure}[h]
\includegraphics[width=0.8 \textwidth]{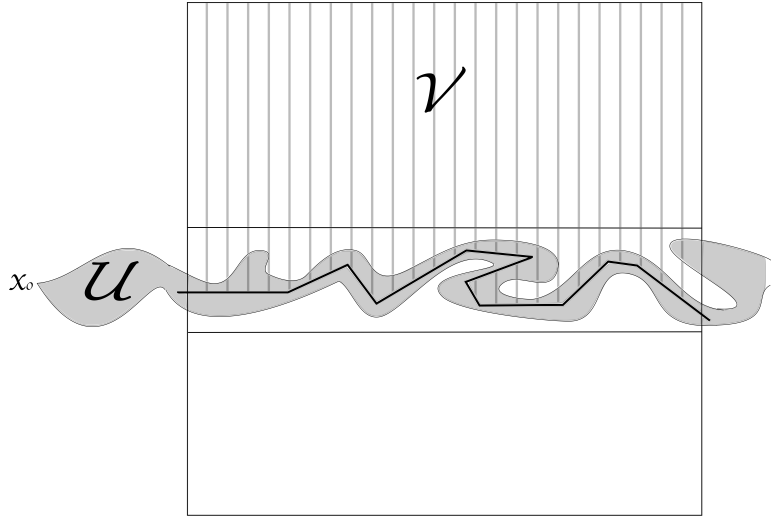}
\caption{Region $\mathcal V$}
    \label{fig1}
\end{figure}

Define $w$ is $R$ as
\be\label{w}  w:=\left\{ \begin{array}{l}
u \quad \mbox{in} \quad R \setminus \mathcal V, \\
u \quad \mbox{in} \quad \mathcal V \cap \mathcal U,\\
p \quad \mbox{in} \quad \mathcal V \setminus U.
\end{array}
\right.
\ee
We claim that $w$ is a supersolution. Indeed, notice that $w$ is  a continuous function and it is a solution away from the set $\p \mathcal U \cap \mathcal V$. On the other hand in this set $w=p$ and $w \le p$ in a neighborhood of it, and the claim is proved.

However, $w$ cannot be a supersolution since it stays a bounded distance away from the function $a x_2^+$ with $a \gg 1$, which is a strict subsolution. 

Precisely we have
\be\label{wa}
|w-a x_2^+| \le 2.
\ee
Let 

\be\label{cs}
\varphi=x_2-x_1^2+ C(\lambda,\Lambda) (x_2-x_1^2)^2,
\ee
with $C$ large such that $\varphi$ is a subsolution for both $L_1$, $L_2$. Then
$$\Psi:=\omega(\gamma) \varphi^+ - \gamma \varphi^- $$ is a (classical) subsolution for the two-phase 
problem in a fixed neighborhood of $0$. Here $\omega$ is defined in \eqref{omega}. From \eqref{wa} we may choose $\gamma$ large, universal such that 
a translation of the graph of $\Psi$ is tangent by below at an interior point to the graph of $w$, provided that $a$ is sufficiently large, and we reach a contradiction. 
\end{proof}

\begin{rem}\label{r1}
In the proof above we used the thin component $\mathcal U$ which concentrates near the line $\{x_2=0\}$ in order to {\it glue} the solution $u$ on one side of $\mathcal U$ with the two-plane solution $p$ on the other side of $\mathcal U$ and obtain a supersolution. One can also obtain a subsolution if in the glueing region $\mathcal U$ we replace $u$ by $p$ i.e.
\be\label{v}v:=\left\{ \begin{array}{l}
u \quad \mbox{in} \quad R \setminus (\mathcal V \cup \mathcal U), \\
p \quad \mbox{in} \quad \mathcal V \cup \mathcal U
\end{array}
\right.
\quad \quad \quad \mbox{is a subsolution.}
\ee

Assume now that a sequence $v_m$ of solutions converges uniformly in ball $B_r(z)$ to a function $v^*$ and $p_m$ a sequence of two-plane solutions converges to $p_0$. If there are connected components of $\{v_m<p_m\}$ which converge in the Hausdorff distance to a $C^1$ curve $$\Sigma \subset \{v^*=p_0\} \cap B_r(z),$$ then using supersolutions and subsolutions as in \eqref{w}-\eqref{v}, we see that the function which is $v^*$ on one side of $\Sigma$ and $p_0$ on the other side of $\Sigma$ is also a solution. By the unique continuation result below we deduce that $v^*$ coincides with $p_0$. 
\end{rem}

\begin{lem}[Unique continuation] \label{uc}
Assume that $u$ satisfies \eqref{FBC} and $u=p$ in an open set, where $p$ is a two-plane solution. Then $u \equiv p$. 
\end{lem}

\begin{proof}
Assume that $u=p=ax_2^+-bx_2^-$, (with $a,b>0$, $a=G(b,e_2)$) in a neighborhood of some point in $\{x_2<0\}$. Then by unique continuation $u=p$ in $\{x_2 \le 0\}$. Since $\{u>0\} \subset \{x_2 > 0\}$ and $u$ is Lipschitz continuous we conclude that $u^+$ has an expansion at $0$ (see Lemma 11.17 \cite{CS})
$$u^+=t x_2^+ +o(|x-y|).$$
Using the free boundary condition at $0$ we find $t=a$. 

Next we claim that $\{u>0\}$ in $\{x_2>0\}$ near $0$. 
Indeed, let $\varphi$ be as in \eqref{cs} and consider the comparison function 
$$\Psi:=(a-\eps) \varphi^+-\beta \varphi^-,$$
with $\beta>0$ small such that $\Psi$ is a comparison subsolution in a small $\sigma$-neighborhood of $0$. Notice that $\Psi$ is strictly increasing in the $e_2$ direction. Using the expansion of $u^+$, we see that in $B^+_\eta$ for $\eta$ sufficiently small, we can compare $u$  with translations of the rescaled subsolutions $$ \frac{\eta}{\sigma}\Psi \left ( \frac \sigma \eta x \right ),$$ and obtain that $u>0$ in $B^+_{\eta/2}$, and the claim is proved.

The free boundary condition gives $u_{e_2}^+=a$ on $\{x_2=0\} \cap B^+_{\eta/2}$ and by unique continuation we obtain that $u$ coincides with $p$ in $B^+_{\eta/2}$ and therefore in $\{x_2>0\}$.

\end{proof}

The proof of Theorem \ref{T2} is a more refined version of the arguments used in Theorem \ref{TLip}. We present it below.

\begin{proof}[\it Proof of Theorem $\ref{T2}$]
We remark that $u^*$ is a solution to our two-phase problem with constant coefficients operators, since it is the uniform limit (on compacts) of a sequence of solutions.

Denote by 
$$D^+=\nabla u^* (\{u^*>0\}), \quad D^-=\nabla u^* (\{u^*<0\}),$$
and $D^\pm$ are bounded sets, since $u^*$ is Lipschitz. Assume that $D^- \ne \emptyset$ otherwise alternative $(b)$ holds and there is nothing to prove. We choose a direction from the origin, say $e_2$ for simplicity, which intersects $D^-$. We let
\be\label{b}b : = \max \{ t | \quad t e_2 \in \overline {D^-}  \quad \mbox{or} 
\quad G(t,e_2) e_2 \in \overline {D^+} \}, \quad \quad a:=G(b,e_2),\ee
and then $b>0$, hence $a>0$ as well. 

Without loss of generality we assume that $a=b=1$, since we can multiply $G$ by a constant so that $G(1,e_2)=1$, and let $$p_0:=x_2.$$

The definition of $a$, $b$ above says that there exists a sequence of points $x_k$ such that the 
corresponding two-plane solution for $u^*$ at $x_k$ has normal $\nu_k$, and slopes 
$a_k$, $b_k$ with $\nu_k \to e_2$, $b_k \to 1$, $a_k \to 1$. Moreover there are no points 
for which the two-plane solution has normal $e_2$ and slopes strictly bigger than $1$.

Theorem \ref{T2} will follow easily from a unique continuation argument and the open mapping theorem, once we establish the next lemma which says that the slope $e_2$ is in fact achieved at some point in $\{u^*>0\} \cup \{u^*<0\}$.

\begin{lem}\label{l2.4}
$$e_2 \in D^+ \cup D^-.$$
\end{lem}

\begin{proof}

Assume by contradiction that the conclusion does not hold. Then $u^*$ and $p_0$ cross transversely away from the zero level set, thus $\{ u^*= p_0\}$ consists of a union of non self-intersecting $C^2$ curves away from $\{x_2=0\}$.

Notice that any connected component 
of $\{ u^* <p_0\}$ must be unbounded.

We now divide the proof of the Lemma in three steps.

\medskip

{\bf Step 1.} We prove a statement about the connected components of $\{ u^* < p_0\}$. We show that only a finite number of the components of $\{u^*< p_0\} \cap B_1$ intersect $B_{1/2}$, and moreover only one of these components intersects a small neighborhood of $0$.

\begin{lem}\label{l3}
$\{u^*< p_0 \} \cap B_1$ has only one connected component near $0$.
\end{lem}

In other words, there is a ball $B_\delta$ such that all the points in $B_\delta \cap \{u^*<p_0\}$ can be joined by a path in $\{u^*< p_0 \} \cap B_1$.

\begin{proof} We distinguish two cases.

{\it Case 1:} There exists $r \in (1/2,1)$ such that $u^*(\pm r,0) \ne 0$. Let
$$\mathcal R:=[-r,r] \times [-\delta, \delta]$$
be a rectangle with $\delta$ small such that on each lateral side of $\mathcal R$ we have either $\{ u^*>p_0\}$ or $\{ u^*< p_0\}$. Denote by $T^+$ and $T^-$ the top and bottom sides of $\mathcal R$,
$$T^+ =[-r,r] \times \{\delta \}, \quad \quad T^- =[-r,r] \times \{-\delta \}.$$
We remark that, since $u^*$ and $p_0$ cross transversally away from the line $x_2=0$, there are only a finite number of components of $\{u^*<p_0\} \cap B_1$ which intersect $\p \mathcal R$.

\

{\it Claim:} Each connected component $\mathcal U$ of $\{u^*<p_0 \} \cap \mathcal R$ must intersect $T^+$.

\

If $\mathcal U$ intersects a lateral side of $\mathcal R$ then it intersects $T^+$ as well. 
Assume that $\mathcal U$ intersects only the bottom boundary $T^-$ of $\mathcal R$ (as in the picture below, Figure \ref{fig2}). 

\begin{figure}[h]
\includegraphics[width=0.7 \textwidth]{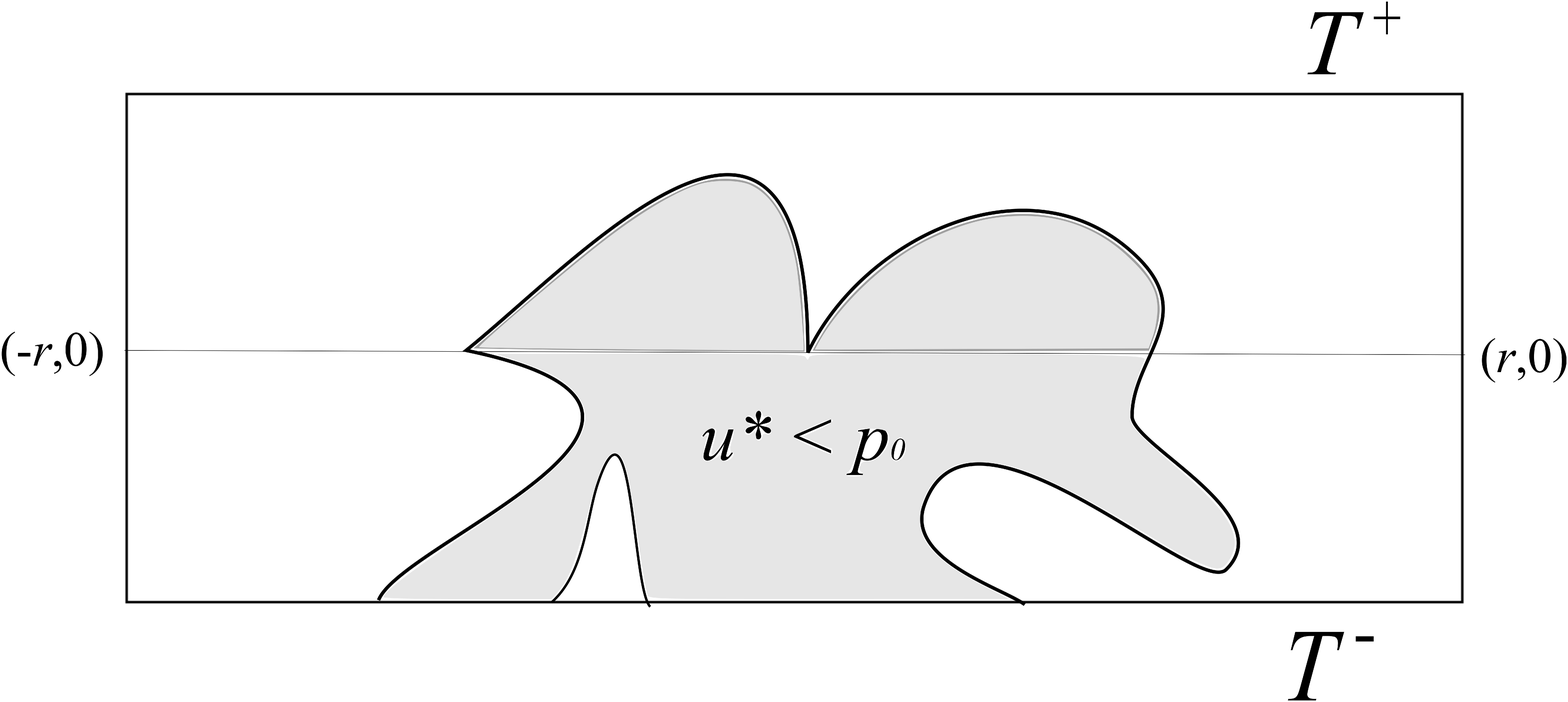}
\caption{}
    \label{fig2}
\end{figure}

Consider the family of linear functions $$\ell_s:=(1+ s) x_2 - 2s, \quad \quad \mbox{ for $s>0$,}$$ and notice that in $B_1$we have $p_0>\ell_s$ for all $s>0$ and $u^*> l_s$ for $s$ large. We decrease $s$ till the graph of $l_s$ touches the graph of $u^*$ by below at a point $(x^*, u^*(x^*)).$

If $x^* \in T^-$ then we find $\nabla u^*(x^*)=te_2$ for some $t \ge 1+s >1$, and we contradict that $b=G(1,e_2)=1$, with $b$ defined in \eqref{b}. 
 
Similarly,  if $x^* \in \{x_2 > -\delta\}$ and $u^*(x^*) \ne 0$ then $\nabla u^*(x^*)=(1+s)e_2$, again a contradiction. 
 
 If $u^*(x^*)=0$ then $(u^*)^+$ has a linear expansion near $x^*$ 
$$(u^*)^+(x)=t(x_2-x_2^*)^+ +o(|x-x^*|), \quad \quad \mbox{for some $t\ge 1+s$}.$$
 This implies $t e_2 \in \overline {D^+}$ and we reach a contradiction as above, and the claim is proved.

 Analogously, we obtain that each component of $\{u^*>p_0\} \cap \mathcal R$ must intersect $T^-$.

Now let us assume by contradiction that two distinct components $\mathcal U_1$ and $\mathcal U_2$ 
of $\{u^*<p_0\} \cap \mathcal R$ intersect any neighborhood of the origin i.e. 
\be\label{U1U2}
0 \in \p \mathcal U_1 \cap \p \mathcal U_2.
\ee 
From the claim we find points $y_i \in \mathcal U_i \cap T^+$. Let $z$ be a point on the segment $[y_1,y_2]$ which belongs to the set $\{u^* = p_0  \}$. Since $u^*$ and $p_0$ cross transversally and $\nabla u^*(z) \ne t e_2$ with $t \ge 1$ we see that $z \in \partial \mathcal V$ for some component $\mathcal V$ of $\{u^* >p_0\} \cap \mathcal R$. Using the claim again, we can find a non self-intersecting polygonal line which connects $z$ with the bottom $T^-$ and which is included in $\mathcal V$ except at the initial point $z$. This polygonal line splits the rectangle $\mathcal R$ into two separate regions one containing $\mathcal U_1$ and the other $\mathcal U_2$ and we contradict \eqref{U1U2}.

 \

 {\it Case 2:}  $u^*=0$ on a segment $\ell \subset \{x_2=0\}$.
 
 We show that there exists a point $x_0 \in \ell$ and $\delta>0$ small such that
\be\label{st2} 
u^*< p_0 \quad \mbox {in $B^+_\delta(x_0)$}, \quad \quad u^*> p_0 \quad \mbox {in $B^-_\delta(x_0)$}.
\ee
 Once we establish this, the arguments from {\it Case 1} carry through as before. Indeed if \eqref{st2} holds, we can construct a rectangle $\mathcal R$ for which the claim holds since we can guarantee that if $\mathcal U$ intersects a lateral side of $\mathcal R$ then it intersects $T^+$ as well.
 
We show only the first inequality in \eqref{st2} since the second follows in the same way. 
The function $(u^*)^+$ has an expansion in $\{x_2>0\}$ near a point $y \in \ell$ as
 $$(u^*)^+=t x_2^+ +o(|x-y|).$$
 
 If $t<1$ then we can compare $u^*$ with its {\it harmonic replacement} in $B_\eta^+(y)$ for some small $\eta$ and we easily obtain that $(u^*)^+ \le \frac{t+1}{2}x_2<(1-\eps)p_0$ in $B_{\eta/2}^+(y)$.
 
 If $t \ge 1$ then as in the proof of Lemma \ref{uc} and find that $u^*>0$ in $B^+_\eta(y)$. 

As before, if $(u^*)^+_\nu <1$ at any point on $\ell \cap B_{\eta}(y)$, then we obtain the desired conclusion in a neighborhood of that point. If $(u^*)^+_\nu >1$ at any point on the same segment then we find $\tilde t e_2 \in \overline{D^+}$ for some $\tilde t>1$ and we reach a contradiction. Otherwise, $(u^*)^+_\nu =1$ on the whole segment, and by unique continuation we obtain that $u^*$ coincides with $x_2$ in $B_{\eta/4}^+(y)$, thus $e_2 \in D^+$, contradiction. 

In conclusion \eqref{st2} holds and Lemma \ref{l3} is proved.

\end{proof}

{\bf Step 2.} From the definitions of $u^*$, $D^\pm$ we easily obtain.

\begin{lem}\label{l4}
There exists a sequence of points $y_m \to 0$ and blow-up functions $u_m$ such that $D^2 u_m(y_m) \ne 0$ and $\nabla u_m(y_m) \to e_2$. This means that the two-plane solutions $p_m$ corresponding to $y_m$ satisfy $p_m \to p_0$.  
\end{lem}

Let $x_k$ be a point such that $u^*(x_k) \ne 0$, $\nabla u^*(x_k)$ is close to $e_2$ and $D^2u^*(x_k) \ne 0$. Since
$$u_m(x):= \frac {1}{r_m} u(r_m x)  \quad \quad r_m \to 0,$$
converge uniformly (in $C^2$) to $u^*$ in a neighborhood of $x_k$, we find that $y_m:=r_m x_k$ and $u_m$ for some large $m$ depending on $k$ has the required properties of the lemma.

\medskip

{\bf Step 3.} We apply the two components lemma \ref{l2} for $u_m$ at $y_m$ and recall that $u_m \to u^*$, $p_m \to p_0$, $y_m \to 0$. In view of Lemma \ref{l3} one of the two connected components of $\{u_m<p_m\}$ must concentrate as $m \to \infty$ near the set $\{u^*=p_0\}$ in a small ball $B_\delta$ around the origin. Next we want to apply Remark \ref{r1} and reach a contradiction. As in Lemma \ref{l3} we consider the two cases.

{\it Case 1}:  There exists $r \in (0,\delta)$ such that $u^*(\pm r,0) \ne 0$. Then we can find a point $z \in \p B_r \cap \{u^* =p_0 \}$ and a neighborhood of $z$ away from $\{x_2=0\}$ where we can apply Remark \ref{r1}. 

\smallskip

{\it Case 2}:  $\{u^*=p_0\}$ contains a segment on $x_2=0$ and this segment is the limit of a sequence of connected components of $\{u_m<p_m\}$. Then Remark \ref{r1} applies again at some interior point $z$ of this segment.
\end{proof}

\smallskip

{\it End of proof of Theorem $\ref{T2}.$}
In conclusion $e_2 \in D^+ \cup D^-$, and there exists $x_0$ such that $\nabla u^*(x_0)=e_2$. Since we are in 2D, in view of the open mapping Theorem and the definition \eqref{b} of $b=1$, we conclude that $\nabla u^* \equiv e_2$ in a neighborhood of $x_0.$
Hence $u^*$ is linear in such neighborhood and by Lemma \ref{uc} we conclude that $u^* \equiv p_0$. 

 \end{proof}
  
\section{Compactness arguments}

The purpose of this section is to obtain Theorem \ref{T3}. The key ingredient is Proposition \ref{p2} below. First, we observe that
by the method of \cite{DFS1,DFS2} (which applies also for two different operators) we have the following perturbation result which holds in any dimension.
  
\begin{prop}\label{p1}
Assume that $p=p_{0, \nu,a,b}$ is a two-plane solution and $G$ is $C^1$ in a neighborhood of $(b,\nu)$. There exists $\eps_0$ small depending only on $G$ such that if $u$ is a solution and
$$|u-p| \le \eps \quad \mbox{in $B_1$,}  $$
for some $\eps \le \eps_0$, then $F(u)$ is a $C^{1,\beta}$ graph in $B_{1/2}$ with $C^{1,\beta}$ norm bounded by $C \eps$.   
\end{prop}

Indeed, the linearization becomes a transmission problem for different constant coefficients operators across the plane $x \cdot \nu=0$ with the jump condition
\be\label{jump}G(b, \nu) v_\nu^+=v_\nu^- \,  G_1(b,\nu) b + v_\tau \cdot G_{\nu} (b,\nu).\ee

This is obtained formally by expanding in $\eps$  the free boundary condition for the perturbed solution (say $0\in F(u), \nu(0)=e_n$),
$$u= a(x_n + \eps v)^+ - b(x_n + \eps v)^-, \quad a=G(b,e_n).$$ This leads to the jump condition \eqref{jump}.

This linear problem is invariant under translations along directions perpendicular to $\nu$, and therefore solutions are smooth in these directions. A viscosity approach to this type of transmission problem has been developed in \cite{DFS2} (see Theorem 3.2 and the main Proposition 3.5.) In that context $G$ did not depend on $\nu$. However, all arguments are easily adapted to the case when $G$ depends also on $\nu$ (see also \cite{AM}.) 

Next we assume that $G$ is homogenous of degree 1 in the $b$ variable. In view of Theorem \ref{T2} and Proposition \ref{p1} we obtain
\begin{prop}\label{p2}
Assume that $u$ is a solution in $B_1 \subset \R^2$ and $u(0)=0$, $|u| \le 1$. Then there exists a two-plane solution $p$ (here we include also $p \equiv 0$) such that
$$|u-p| \le C_0 |x|^{1+\alpha}$$
for some $C_0$ depending only on $G$.
\end{prop}

\begin{proof}
We know that $u$ is Lipschitz in $B_{1/2}$, i.e. $|\nabla u| \le C$, and we pick $\eps_0$ sufficiently small such that Proposition \ref{p1} applies with this $\eps_0$ for all planes $p$ with either $a$ or $b$ in $[\frac14,C]$. 

{\it Claim:} There exists $\rho\in (0, \frac 12)$ depending only on $G$ such that in some ball $B_{r_0}$ with
$r_0 \in (\rho,\frac 12)$ depending also on $u$ we have 
$$\mbox{either} \quad  |u-p| \le \eps_0 r_0 \quad \quad \mbox{or} \quad |u| \le \frac 12 \,  r_0,$$
for some two-plane solution $p_{a,b}$ with either $a$ or $b$ in $[\frac14,C]$.

\

Indeed if this property does not hold for a sequence of $\rho_k \to 0$ and functions $u_k$, then we can extract a subsequence which converges uniformly to a limiting solution $u_\infty$. By Theorem \ref{T2} we have
$$u_\infty=p_{a,b} + o(|x|).$$
If either $a$ or $b$ are greater than $\frac 14$ then we contradict the first alternative for some $k$ large, and if both $a$, $b$ are less than $1/4$ then we contradict the second alternative, and the claim is proved.

\

We choose $\alpha>0$ small such that $1/2 \le \rho^\alpha$ and $\alpha \le \beta$ with $\beta$ given by Proposition \ref{p1}. Notice that if the first alternative holds then the conclusion of Proposition \ref{p2} is clearly satisfied for some large $C_0$. If the second alternative holds then the rescaling
$$\tilde u:=r^{-(1+\alpha)}u(rx), \quad \quad r=r_0,$$
is a solution in $B_1$ and $|\tilde u|\le 1$. Now we apply the claim to $\tilde u$ and repeat this process either

\

1) a finite number of times with $r=r_k$, $r_{k+1}/r_k \in (\rho,\frac 12)$ and stop the first time we end up with the first alternative for $\tilde u$. Then $$|u-p_{\nu,a,b}| \le C_0|x|^{1+\alpha} \quad \mbox{in $B_r$}, \quad \quad |a|,|b| \le C r^{\alpha},$$
and $|u|,|p|$ are bounded by $C|x|^{1+\alpha}$ outside $B_r$ and the conclusion follows.

\

2) an infinite number of times and $\tilde u$ satisfies the second alternative indefinitely. In this case we clearly satisfy Proposition \ref{p2} with $p \equiv 0$.

\end{proof}
 
\begin{rem}\label{r2}
Proposition \ref{p2} implies that for any point $x_0 \in B_{1/4}$ we have
\be\label{300}
 |u-p_{x_0}| \le C |x-x_0|^{1+\alpha}
\ee
where $p_{x_0}$ is the two-plane solution at $x_0$, and $C$ is a universal constant.

Indeed, let $r$ be the distance from $x_0$ to $\{u=0\}$ and $z$ a point where the distance is realized. Denote by $p_z$ the two plane solution at $z$ given by Proposition \ref{p2}, hence $|u-p_z| \le Cr^{1+\alpha}$ in $B_r(x_0)$. This gives $|p_{x_0}-p_z| \le C r^{1+\alpha}$ in $B_r(x_0)$ which implies
$$|p_{x_0}-p_z| \le  C r^\alpha |x-x_0| \quad \mbox{outside $B_r(x_0)$,}$$
and the claim easily follows.
\end{rem}

\section{Reduction to a two-phase free boundary problem}

In this section we show that problem \eqref{BE} can be reduced to a two-phase problem of the form \eqref{TP}, for a specific $G$ and we finally obtain our main Theorem \ref{T0}.

Without loss of generality we may assume that 
$$L_1v=\triangle v \quad \mbox{ and} \quad L_2 v =v_{11}+ m \, \,  v_{22}, \quad \mbox{for some $m \ge 1$.}$$

We establish the following result.

\begin{prop}\label{P1}
The function $$u= v_{22}$$ satisfies the two-phase free boundary problem
$$L_1u=0 \quad \mbox{in} \quad \{u>0 \} , \quad \quad  L_2u=0 \quad \mbox{in} \quad \{u<0 \}, $$
$$u_\nu^+=(1+ (m-1)\nu_2^2) \,  u_{\nu}^-  \quad \mbox{on $\Gamma$}.$$
\end{prop}

The free boundary condition can be easily deduced when the free boundary $\Gamma \in C^1$. Indeed let $w_1:= L_2 v=(m-1)u^+$, $w_2:=L_1 v=(m-1)u^-$. If $\nu$ denotes the normal to $\Gamma$ then
$$L_1L_2 v = L_i w_i = (w_i)_\nu^+ A_i(\nu,\nu)  \, \, d \mathcal H^{1}|_{\Gamma}.$$  
This computation holds in any dimension and gives the free boundary condition 
$$(w_1)^+_\nu A_1(\nu, \nu) = (w_2)^+_\nu A_2(\nu,\nu) \quad \text{on $\Gamma.$}$$ In fact the proof below of Proposition \ref{P1} applies in any dimension.

We prove Proposition \ref{P1} by approximating $F$ by $C^2$ operators $$F_\eps(D^2u)= u_{11}+h_\eps (u_{22})$$ with $h_\eps(s)$ a smoothing of the Lipschitz function 
$$h_0(s)= s^+ - m \, \, s^-.$$

\begin{proof}

We differentiate twice in the $e_2$ direction and obtain that a solution to $F_\eps(D^2v)=0$ satisfies
$$ \triangle v_{22} + h_\eps'(v_{22})v_{2222} + h_\eps''(v_{22})v_{222}^2=0,$$
hence $u=v_{22}$ satisfies
\be\label{30}
 \triangle u+h_\eps'(u)u_{22} +h_\eps''(u)u_2^2=0.
\ee

By $C^{2,\alpha}$ estimates, the solutions $u$ above are uniformly H\"older, hence converge uniformly to $\bar u$ as $\eps \to 0$. In order to find the 
equation for $\bar u$ we need to find a family of solutions of \eqref{30} (or subsolutions  and 
supersolutions) that converge with $\eps$.

First we look for one-dimensional functions $g_\eps (x \cdot \nu)$ which solve \eqref{30}. For simplicity of notation we drop the subindex from $g$, $h$ and we obtain 
$$\left [1+\nu_2^2\, h'(g) \right] \, g'' + \nu_2^2 \,  h''(g) \, \, g'^2=0,$$
or
$$ \left( [1+\nu_2^2 \, h'(g) ] g' \right)'=0,$$
hence $$\frac {g'(t_2)}{g'(t_1)} = \frac{1+\nu_2^2 \, h'(g(t_1))}{1+\nu_2^2 \, h'(g(t_2))} $$
which means that the derivative of $g$ jumps by a factor of $\nu_1^2+m \nu_2^2$ after passing through $0$. Thus any two-plane function
$$a (x \cdot \nu)^+-b(x \cdot \nu)^-, \quad \quad  \mbox{with} \quad a=(\nu_1^2 +m \nu_2^2) b \ge 0,$$ is the uniform limit of solutions $g_\eps$ and therefore it is a comparison function for $\bar v$.

Now we can slightly modify the comparison function above and obtain a perturbed family for which the free boundary is a large sphere instead of a hyperplane. Consider functions $g_\eps(d)$ where $d$ is the signed distance to a sphere of radius $\delta^{-3}$ passing through the origin and with inner normal $\nu$, with $d>0$ inside the sphere. We let $g(t)$ such that it satisfies the ODE
$$[1+ (\nu_2^2+\delta)h'(g)] g'=1+ \delta t, \quad g(0)=0.$$
Using that $h$ is concave and $h(0)=0$ for $t>0$ we obtain that $g''>c\delta$, $g'<C$ and one can easily see as in the computation above that $g(d)$ is a subsolution in $B_1$.

As $\eps \to 0$, $g_\eps$ converges uniformly to 
$$\bar g(t):= \left(t+ \frac \delta 2 t^2 \right)^+ - b \left(t+ \frac \delta 2 t^2 \right )^- \quad \mbox{with} \quad b=\frac{1}{\nu_1^2 + m \nu_2^2} +O(\delta),$$
and $\bar g$ is a comparison subsolution for $\bar v$. We obtain the desired conclusion since any comparison subsolution $a \varphi^+-b\varphi^-$ with $\nabla \varphi(0)=\nu$ and $a>G(b, \nu)$ appearing in the viscosity definition of \eqref{FBC} can be touched strictly by below at the origin by a multiple and a rescaling of $\bar g(d)$ (with $\delta$ sufficiently small).
\end{proof}

The proof of our main Theorem \ref{T0} now follows immediately.

\smallskip

\noindent {\it Proof of Theorem $\ref{T0}$}
From Proposition \ref{p1} and Remark \ref{p2} we know that $u:=v_{22}$ satisfies \eqref{300}. Then $$L_2 v=(m-1)u^+=:w$$
and for each $x_0$ in $B_{1/4}$ there exists a point $z \in \{ u=0\}$ such that 
$$ |w- \sigma ((x-z) \cdot \nu)^+ | \le C |x-x_0|^{1+\alpha}, \quad \mbox{for some $\sigma \ge 0$.} $$
Let $q=\gamma[(x-z)  \cdot \nu)^+]^3$ be such that $L_2 q=\sigma ((x-z) \cdot \nu)^+$ and then the right hand side of $L_2(v-q)$ is pointwise $C^{1,\alpha}$ at $x_0$. By the pointwise Schauder estimates for linear equations we obtain that $v-q$ is pointwise $C^{3+\alpha}$ at $x_0$.

\qed


\begin{thebibliography}{9999}
\bibitem[ACF]{ACF} Alt H.W., Caffarelli L.A., Friedman A., \emph{Variational problems with two phases and their free boundaries. }Trans. Amer. Math. Soc. {\bf282} (1984), no. 2, 431--461.
\bibitem[AM]{AM}  Andersson J., Mikayelyan H., { \it The zero level set for a certain weak solution, with applications to the Bellman equations,} Trans. Amer. Math. Soc. 365 (2013), no. 5, 2297--2316. 
\bibitem[B]{B} Bernstein S.,
{\it Uber ein geometrisches Theorem und seine Anwendung auf die partiellen Dif-
ferentialgleichungen vom elliptischen Typus,}
Math. Z.,
26
(1927), 551--558.
\bibitem[BE]{BE}Brezis H., Evans L.C., A variational inequality approach to the Bellman-Dirichlet equation for two elliptic operators, Arch. Rational Mech. Anal. 71 (1979) 1--13.
\bibitem[CJK]{CJK} Caffarelli L.A.,  Jerison D., Kenig C.E., {\it Some new monotonicity theorems with applications to free boundary problems}, Ann. of Math. (2) {\bf 155} (2002), no. 2, 369--404.
\bibitem[CS]{CS}Caffarelli L.,  Salsa S., 
{\it A geometric approach to free boundary problems,}
Graduate Studies in Mathematics, 68. American Mathematical Society, Providence, RI, 2005. x+270 pp.
 \bibitem[DFS1]{DFS1} De Silva D., Ferrari F., Salsa S., {\it Two-phase problems with distributed source: regularity of the free boundary,}  Anal. PDE 7 (2014), no. 2, 267--310.
\bibitem[DFS2]{DFS2}  De Silva D., Ferrari F., Salsa S., {\it Free boundary regularity for fully nonlinear non-
homogeneous two-phase problems,} Journal de Mathematiques Pures et
Appliquees 103 (2015), 658--694.
\bibitem[DK]{DK} Dipierro S., Karakhanyan A.,  {\it Stratification of free boundary points for a two-phase variational problem},  arXiv:1508.07447
\bibitem[DS]{DS}  De Silva D.,  Savin O.,  {\it Minimizers of convex functionals arising in random surfaces,} Duke Math. J. 151 (2010), no. 3, 487--532.
\bibitem[E]{E} Evans L.C., {\it Classical solutions of fully nonlinear, convex, second-order elliptic equations,} Comm. Pure Appl. Math. 25 (1982) 333--363.
\bibitem[F]{F} Feldman M., {\it Regularity for nonisotropic two-phase problems with Lipschitz free boundaries}, Differential Integral Equations 10 (1997), no.6, 1171--1179.
\bibitem[H]{H} Hopf E.,
{\it On S. BernsteinÕs theorem on surfaces
$z(x,y)$
of nonpositive curvature,}
Proc. Amer.
Math. Soc.,
1
(1950), 80--85.
\bibitem[K]{K} Krylov N.V., Boundedly inhomogeneous elliptic and parabolic equations, Izv. Akad. Nauk SSSR 46 (1982) 487--523 (in Russian); English transl.: Math. USSR-Izv. 20 (1983) 459--492.
\bibitem[MP]{MP} Matevosyan N., Petrosyan A., {\it Almost monotonicity formulas for elliptic and parabolic operators with variable coefficients}, Comm. Pure Appl. Math 44 (12011), 271-311.
\bibitem[S]{S}  Savin O., {\it $C^1$ regularity for infinity harmonic functions in two dimensions}, Arch. Ration. Mech. Anal. 176 (2005), no. 3, 351--361. 

\end{thebibliography}
\end{document}